\documentclass[11pt]{article}
\usepackage{amsmath, amssymb,verbatim,appendix}
\usepackage[mathscr]{eucal}
\usepackage{amscd}
\usepackage{amsthm}
\usepackage{stmaryrd}
\usepackage{enumerate}
\usepackage{comment}
\usepackage[latin1]{inputenc}
\usepackage{tikz}
\usetikzlibrary{shapes,arrows}
\usepackage{cite}
\usepackage{url}
\usepackage[margin=1in]{geometry}

\AtBeginDocument{%
   \def\MR#1{}
}

\newtheorem{theorem}{Theorem}
\newtheorem{lemma}{Lemma}

\newtheorem{question}{Question}

\newtheorem{problem}{Problem}
\newtheorem{observation}{Observation}

\newtheorem{claim}{Claim}

\newtheorem*{theorem2}{Theorem 2}
\newtheorem*{theorem3}{Theorem 3}
\newtheorem*{theorem4}{Theorem 4}

\newtheorem*{lemma1}{Lemma 1}

\theoremstyle{definition}
\newtheorem{example}{Example}
\newtheorem{definition}{Definition}

\newcommand{\abar}{\bar{a}}
\newcommand{\bbar}{\bar{b}}
\newcommand{\cbar}{\bar{c}}

\newcommand{\ubar}{\bar{u}}
\newcommand{\vbar}{\bar{v}}

\newcommand{\xbar}{\bar{x}}
\newcommand{\ybar}{\bar{y}}
\newcommand{\zbar}{\bar{z}}

\newcommand{\Cbar}{\bar{C}}

\newcommand{\VC}{\textnormal{VC}}

\newcommand{\calL}{\mathcal{L}}
\newcommand{\calF}{\mathcal{F}}
\newcommand{\calH}{\mathcal{H}}

\newcommand{\calM}{\mathcal{M}}

\newcommand{\calB}{\mathcal{B}}

\newcommand{\calP}{\mathcal{P}}

\title{
$\VC_{\ell}$-dimension and the jump to the fastest speed of a hereditary $\calL$-property
}

\author{
C. Terry}
\date{}

\begin{document}

\maketitle

\begin{abstract}
In this paper we investigate a connection between the growth rates of certain classes of finite structures and a generalization of $\VC$-dimension called $\VC_{\ell}$-dimension. Let $\calL$ be a finite relational language with maximum arity $r$.  A hereditary $\calL$-property is a class of finite $\calL$-structures closed under isomorphism and substructures.  The \emph{speed} of a hereditary $\calL$-property $\calH$ is the function which sends $n$ to $|\calH_n|$, where $\calH_n$ is the set of elements of $\calH$ with universe $\{1,\ldots, n\}$. It was previously known there exists a gap between the fastest possible speed of a hereditary $\calL$-property and all lower speeds, namely between the speeds $2^{\Theta(n^r)}$ and $2^{o(n^r)}$.  We strengthen this gap by showing that  for any hereditary $\calL$-property $\calH$, either $|\calH_n|=2^{\Theta(n^r)}$ or there is $\epsilon>0$ such that for all large enough $n$, $|\calH_n|\leq 2^{n^{r-\epsilon}}$.  This improves what was previously known about this gap when $r\geq 3$.  Further, we show this gap can be characterized in terms of $\VC_{\ell}$-dimension, therefore drawing a connection between this finite counting problem and the model theoretic dividing line known as $\ell$-dependence.

\end{abstract}

%************************************************************************
\section{Introduction}

%************************************************************************

One of the major themes in model theory is the search for dividing lines among first order theories.  The study of dividing lines was first developed by Shelah \cite{classification}.  One of the main goals of this work was to understand the function $I(T,\kappa)$, which, given an input theory $T$ and a cardinal $\kappa$, outputs the number of non-isomorphic models of $T$ of size $\kappa$. Therefore, the discovery of dividing lines was fundamentally related to infinitary counting problems.  Further, many dividing lines can be characterized by a counting dichotomy, including stability, NIP, VC-minimality, and $\ell$-dependence.   These facts show us that model theoretic dividing lines are closely related to counting problems in the infinite setting.

There has been substantial work on understanding dichotomies in finitary counting problems in the field of combinatorics, particularly in the setting of graphs.  A \emph{hereditary graph property} is a class of finite graphs $\calH$, which is closed under isomorphism and induced subgraphs.  Given a hereditary graph property, $\calH$, the \emph{speed} of $\calH$ is the function $n \mapsto |\calH_n|$, where $\calH_n$ denotes the set of elements in $\calH$ with vertex set $[n]:=\{1,\ldots, n\}$.  The possible speeds of hereditary graph properties are well understood.  In particular, their speeds fall into discrete growth classes, as summarized in the following theorem.

\begin{theorem}\label{gg}
Suppose $\calH$ is a hereditary graph property.  Then one of the following holds, where $\calB_n\sim (n/\log n)^n$ denotes the $n$-th Bell number.
\begin{enumerate}
\item There are rational polynomials $p_0,\ldots, p_k$ such that for sufficiently large $n$, $|\calH_n|=\sum_{i=0}^k p_i(n)i^n$,
\item There exists an integer $k>1$ such that $|\calH_n|=n^{(1-\frac{1}{k}+o(1))n}$, 
\item There is an $\epsilon>0$ such that for sufficiently large $n$, $\calB_n \leq |\calH_n|\leq 2^{n^{2-\epsilon}}$,
\item There exists an integer $k> 1$ such that $|\calH_n|=2^{(1-\frac{1}{k}+o(1))n^2/2}$.
\end{enumerate}
\end{theorem}

This theorem is the culmination of many authors' work.  We direct the reader to \cite{BBW1} for the gap between cases 1 and 2 and within 2, to  \cite{BBW1, ajump} for the gap between cases 2 and 3, to \cite{BoTh2, ABBM} for the gap between 3 and 4, and to \cite{BoTh2} for the gaps within case 4.  Further, it was shown in \cite{BBW2} that there exist hereditary graph properties whose speeds oscillate between the lower and upper bound of case 3, therefore ruling out any more gaps in this range.  Thus Theorem \ref{gg} solves the problem of what are the possible speeds of hereditary graph properties.

On the other hand, there remain many open questions around generalizing Theorem \ref{gg}, even to the setting of $r$-uniform hypergraphs, when $r\geq 3$.  We focus on one such problem in this paper.  If $\calH$ is a hereditary property of $r$-uniform hypergraphs, then $|\calH_n|\leq 2^{n\choose r}$, and it was shown in  \cite{Alekseev1} and \cite{BoTh1} that either $|\calH_n|=2^{cn^r+o(n^r)}$ for some $c>0$, or $|\calH_n|\leq 2^{o(n^r)}$.  In other words, the fastest possible speed of a hereditary property of $r$-uniform hypergraphs is $2^{\Theta(n^r)}$, and there is a gap between the fastest and penultimate speeds.  However, it remained open whether this gap could be strengthened in analogy to the gap between cases 3 and 4 in Theorem \ref{gg}, as we summarize below in Question \ref{ab}.

\begin{question}\label{ab}
Suppose $r\geq 3$.  Is it true that for any hereditary property $\calH$ of $r$-uniform hypergraphs, either $|\calH_n|=2^{cn^r+o(n^r)}$ for some $c>0$, or there is $\epsilon>0$ such that for all large $n$, $|\calH_n|\leq 2^{n^{r-\epsilon}}$?
\end{question}

Given that model theoretic dividing lines are connected to infinitary counting problems, it is natural to ask whether they are also connected to finitary counting problems such as Question \ref{ab}.  The main results of this paper will establish such a connection, as well as answer Question \ref{ab} in the affirmative.  

Given a finite relational language $\calL$, a \emph{hereditary $\calL$-property} is a class of finite $\calL$-structures,  $\calH$, closed under isomorphism such that if $A$ is a model theoretic substructure of $B$ and $B\in \calH$, then $A\in \calH$.  The \emph{speed} of $\calH$ is the function $n \mapsto |\calH_n|$, where $\calH_n$ denotes the set of elements in $\calH$ with universe $[n]$.  The general problems we are interested in are the following.
\begin{itemize}
\item What are the jumps in speeds of hereditary $\calL$-properties?
\item  Can these jumps be characterized via model theoretic dividing lines?  
\end{itemize}
In this paper, we make progress on these problem by improving the known the gap between the penultimate and fastest possible speeds of a hereditary $\calL$-property, and by connecting this gap to the model theoretic dividing line of $\ell$-dependence. Specifically, we will characterize this gap in terms of a cousin of $\VC$-dimension, which we denote $\VC_{\ell}^*$-dimension.  We now state our main result, Theorem \ref{th2}.  We will then discuss how it improves known results and how it is connected to $\ell$-dependence.

\begin{theorem}\label{th2}
Suppose $\calL$ is a finite relational language of maximum arity $r\geq 1$, and $\calH$ is a hereditary $\calL$-property. Then either
\begin{enumerate}[(a)]
\item $VC^*_{r-1}(\calH)<\infty$ and there is an $\epsilon>0$ such that for sufficiently large $n$, $|\calH_n|\leq 2^{n^{r-\epsilon}}$, or
\item $VC^*_{r-1}(\calH)=\infty$, and there is a constant $C>0$ such that $|\calH_n|= 2^{Cn^{r}+o(n^r)}$.
\end{enumerate}
When $r=1$, the following stronger version of (a) holds: $VC^*_{0}(\calH)<\infty$ and there $K>0$ such that for sufficiently large $n$, $|\calH_n|\leq n^K$.
\end{theorem}

Theorem \ref{th2} strengthens what was previously shown in \cite{CTHParxiv}, that for any hereditary $\calL$-property $\calH$, either $|\calH_n|=2^{Cn^r+o(n^r)}$ for some $C>0$, or $|\calH_n|\leq 2^{o(n^r)}$, where $r$ is the maximum arity of the relations in $\calL$.  This result generalizes the gap between cases 3 and 4 in Theorem \ref{gg}, and is new in all cases where $r\geq 3$.  Theorem \ref{th2} answers Question \ref{ab} in the affirmative.

Theorem \ref{th2} also shows the gap between the penultimate and fastest possible speeds of a hereditary $\calL$-property  is characterized by a model theoretic dividing line.  The dimension appearing in Theorem \ref{th2}, $\VC_{\ell}^*$-dimension, is a dual version of the existing model theoretic notion of $\VC_{\ell}$-dimension (see Section \ref{bgd} for precise definitions).  $\VC_{\ell}$-dimension is a direct generalization of $\VC$-dimension defined in terms of shattering ``$\ell$-dimensional boxes."  This dimension was first introduced  in \cite{Shelahstronglydep}, where it is used to define the dividing line called $\ell$-dependence.  $\VC_{\ell}$-dimension and $\ell$-dependence have since been studied from the model theoretic point of view in \cite{hempel, ozlem, CPT, 2depgps, depdreams}.  We will show that the condition $\VC_{\ell}^*(\calH)<\infty$ is a natural analogue of $\ell$-dependence for a hereditary $\calL$-property $\calH$. Thus Theorem \ref{th2} can be seen as characterizing a gap in possible speeds of hereditary $\calL$-properties using a version of the model theoretic dividing line of $\ell$-dependence.

Our next result shows that the gap between polynomial and exponential growth is always characterized by $\VC_0^*$-dimension, regardless of the arity of the language.
\begin{theorem}\label{zerodim}
Suppose $\calL$ is a finite relational language, and $\calH$ is a hereditary $\calL$-property. Then either
\begin{enumerate}[(a)]
\item $\VC^*_{0}(\calH)<\infty$ and there $K>0$ such that for sufficiently large $n$, $|\calH_n|\leq n^K$, or 
\item $\VC^*_{0}(\calH)=\infty$, and there is a constant $C>0$ such that for sufficiently large $n$, $|\calH_n|\geq 2^{Cn}$.
\end{enumerate}
\end{theorem}

Theorem \ref{zerodim} is new at this level of generality, in the labeled setting.  There exist general results on the polynomial/exponential counting dichotomy in the unlabelled setting (see for instance \cite{pouzet1, oudrar}), and it is possible the machinery developed in that line of work could be used to obtain the dichotomy of Theorem \ref{zerodim}.  The connection this paper makes between this problem and $\VC_{\ell}$-dimension is new.  Thus, while the existence of the dichotomy described by Theorem \ref{zerodim} is not surprising given past results, Theorem \ref{zerodim} draws a connection to $\VC_{\ell}$-dimension which we think is important for understanding the larger pattern at work.

The dichotomies in Theorems \ref{th2} and \ref{zerodim} depend on whether $\VC_{\ell}$-dimension is finite or infinite, for certain values of $\ell$.  Both results use the following theorem, which shows that infinite $\VC^*_{\ell}$-dimension always implies a lower bound on the speed.

\begin{theorem}\label{th5}
Suppose $\calL$ is a finite relational language of maximum arity $r$, and $\calH$ is a hereditary $\calL$-property.  If $1\leq \ell \leq r$ and $VC^*_{\ell-1}(\calH)=\infty$, then there is $C>0$ such that for large $n$, $|\calH_n|\geq 2^{Cn^{\ell}}$.
\end{theorem}

Somewhat surprisingly, the converse of Theorem \ref{th5} fails.  In particular, we will give an example of a hereditary property of $3$-uniform hypergraphs with $\VC_1(\calH)<\infty$ but with $|\calH_n|\geq 2^{Cn^2}$ for some $C>0$ (see Example \ref{ex1}). We would like to thank D. Mubayi for bringing said example to our attention.  These observations suggest the following interesting open problem.

\begin{problem}
Suppose $\calL$ is a finite relational language of maximum arity $r\geq 3$ and $\ell$ is an integer satisfying $2\leq \ell <r$.  Say a hereditary $\calL$-property $\calH$ has \emph{fast $\ell$-dimensional growth} if $|\calH_n|\geq 2^{\Omega(n^{\ell})}$.  Characterize the hereditary $\calL$-properties with fast $\ell$-dimensional growth.
\end{problem}

We end this introduction with a brief outline of the paper.  In Section \ref{bgd} we give background on $\VC_{\ell}$-dimension and $\VC_{\ell}^*$-dimension. In Section \ref{tech} we present technical lemmas needed for the proofs of our main results. In Section \ref{secth5} we prove Theorems \ref{th2} and \ref{zerodim} and present Example \ref{ex1}.  In Section \ref{eqsection}, we prove that when $\ell>0$, $\VC_{\ell}^*(\calH)=\infty$ if and only if $\VC_{\ell}(\calH)=\infty$.  

\section{Preliminaries}\label{bgd}

In this section, we introduce $\VC_{\ell}$-dimension for $\ell \geq 1$ and $\VC_{\ell}^*$-dimension for $\ell\geq 0$.  For this section, $\calL$ is some fixed language.  We will denote $\calL$-structures with script letters, e.g. $\calM$, and their universes with the corresponding non-script letters, e.g. $M$.  Given an integer $n$, $[n]:=\{1,\ldots, n\}$.  If $X$ is a set, ${X\choose n}=\{Y\subseteq X: |Y|=n\}$, and if $\xbar=(x_1,\ldots, x_s)$ is a tuple, then $|\xbar|=s$.  

\subsection{$\VC$-dimension and $\VC_{\ell}$-dimension}

In this subsection we define $\VC$-dimension and $\VC_{\ell}$-dimension.  We begin by introducing $\VC$-dimension.  Given sets $A\subseteq X$, $\calP(X)$ denotes the power set of $X$.  If $\calF\subseteq \calP(X)$, then $\calF \cap A$ denotes the set $\{F\cap A: F\in \calF\}$.  We say $A$ is \emph{shattered} by $\calF$ if $\calF \cap A=\calP(A)$.  The \emph{$\VC$-dimension} of $\calF$ is $\VC(\calF)=\sup \{|A|: A\subseteq X \text{ is shattered by }\calF\}$, and the \emph{shatter function} of $\calF$ is defined by $\pi(\calF,m)=\max \{|\calF\cap A|: A\in {X\choose m}\}$.  Observe that $\VC(\calF)\geq m$ if and only if $\pi(\calF,m)=2^m$. One of the most important facts about $\VC$-dimension is the Sauer-Shelah Lemma.

\begin{theorem}[\bf{Sauer-Shelah Lemma}]\label{ssl}
Suppose $X$ is a set and $\calF\subseteq \calP(X)$.  If $\VC(\calF)=d$, then there is a constant $C=C(d)$ such that for all $m$, $\pi(\calF,m)\leq C m^d$.
\end{theorem}

$\VC$-dimension is important in various fields, including combinatorics, computer science, and model theory.  We direct the reader to \cite{simon2015guide} for more details.  Given $\ell \geq 1$, $\VC_{\ell}$-dimension is a generalization of $\VC$-dimension which focuses on the shattering sets of a special form.  If $X_1,\ldots, X_{\ell}$ are sets, then $\prod_{i=1}^{\ell}X_i$ is an \emph{$\ell$-box}.  If $|X_1|=\ldots=|X_{\ell}|=m$, then we say $\prod_{i=1}^{\ell}X_i$ is an $\ell$-box of \emph{height $m$}.  If $X_1'\subseteq X_1, \ldots, X_{\ell}'\subseteq X_{\ell}$, then $\prod_{i=1}^{\ell}X'_i$ is a \emph{sub-box} of $\prod_{i=1}^{\ell}X_i$.

\begin{definition}\label{vcdef1}
Suppose $\ell\geq 1$, $\prod_{i=1}^{\ell}X_i$ is an $\ell$-box, and $\calF\subseteq \calP(\prod_{i=1}^{\ell}X_i)$. The \emph{$\VC_{\ell}$-dimension of $\calF$} is 
$$
\VC_{\ell}(\calF)=\sup \{ m\in \mathbb{N}: \text{ $\calF$ shatters a sub-box of $\prod_{i=1}^{\ell}X_i$ of height $m$}\}.
$$
 The \emph{$\ell$-dimensional shatter function} is $\pi_{\ell}(\calF, m)=\{|\calF \cap A|: A \text{ is a sub-box of $\prod_{i=1}^{\ell}X_i$ of height $m$}\}$.
\end{definition}

$\VC_{\ell}$-dimension was introduced in the model theoretic context in \cite{Shelahstronglydep}, where it was used to define the notion of an $\ell$-dependent theory.  It has since been studied as a diving line in \cite{hempel, ozlem, CPT, 2depgps, depdreams}.  Theorem \ref{CPT}, below, is an analogue of the Sauer-Shelah Lemma for $\VC_{\ell}$-dimension, which was proved in \cite{CPT}.

\begin{theorem}[{\bf Chernikov-Palacin-Takeuchi \cite{CPT}}]\label{CPT}
Suppose $\ell \geq 1$, $Y$ is an $\ell$-box, and $\calF\subseteq \calP(Y)$. If $\VC_{\ell}(\calF)=d<\infty$, then there are constants $C=C(d)$ and $\epsilon=\epsilon(d)>0$ such that for all $m\in \mathbb{N}$, $\pi_{\ell}(\calF,m)\leq C 2^{m^{\ell-\epsilon}}$.
\end{theorem}

We will need more complicated versions of Definition \ref{vcdef1} and Theorem \ref{CPT}. This extra complication comes from the fact that for this paper, we cannot work inside $T^{eq}$, as is done in \cite{CPT} (we will not even be working in a complete theory).  We now fix some notation.  Suppose $X$ is a set, and $k_1,\ldots, k_{\ell}\geq 1$ are integers.  Given $\abar_1\in X^{k_1},\ldots, \abar_{\ell}\in X^{k_{\ell}}$, let $\abar_1\ldots \abar_{\ell}$ denote the element of $X^{k_1+\ldots+k_{\ell}}$ which is the concatenation of the tuples $\abar_1,\ldots, \abar_{\ell}$.  Given nonempty sets $A_1\subseteq X^{k_1},\ldots, A_{\ell}\subseteq X^{k_{\ell}}$, let $A_1\ldots A_{\ell}:= \{\abar_1\ldots \abar_{\ell}: \abar_1\in A_1,\ldots, \abar_{\ell}\in A_{\ell}\}$.  Abusing notation slightly, we will write $\prod_{i=1}^{\ell}A_i$ for the set $A_1\ldots A_{\ell}$.  Observe $\prod_{i=1}^{\ell}A_i\subseteq X^r$,  where $r=k_1+\ldots+k_{\ell}$. We call $\prod_{i=1}^{\ell}A_i$ an \emph{$(\ell,r)$-box} in $X$.  If $|A_1|=\ldots=|A_{\ell}|=m$ for some $m\in \mathbb{N}$, then we say $\prod_{i=1}^{\ell}A_i$ has \emph{height $m$}.  By convention, for $r\geq 1$, a $(0,r)$-box of any height in $X$ is a singleton in $X^r$, and a $(0,0)$-box of any height in $X$ is the empty set.  Given any $0\leq \ell\leq r$, we will say a set $\mathbb{A}$ is an \emph{$(\ell,r)$-box} if there is some set $X$ such that $\mathbb{A}$ is an $(\ell,r)$-box in $X$.

\begin{definition}\label{vcdef2}
Suppose $X$ is a set, $1\leq \ell\leq r$, and $\calF\subseteq \calP(X^r)$. The \emph{$\VC_{\ell}$-dimension of $\calF$} is 
$$
\VC_{\ell}(\calF)=\sup \{ m\in \mathbb{N}: \text{ $\calF$ shatters an $(\ell,r)$-box of height $m$ in $X$}\}.
$$
 The \emph{$\ell$-dimensional shatter function} is $\pi_{\ell}(\calF, m)=\{|\calF \cap A|: A \text{ is an ($\ell$,$r$)-box in $X$ of height $m$}\}$.
\end{definition}

Theorem \ref{CPT} can be directly adapted to these definitions.

\begin{theorem}\label{CPT'}
Suppose $1\leq \ell \leq r$, $X$ is a set and $\calF\subseteq \calP(X^r)$. If $\VC_{\ell}(\calF)=d<\omega$, then there are constants $C=C(d)$ and $\epsilon=\epsilon(d)>0$ such that for all $m$, $\pi_{\ell}(\calF,m)\leq C 2^{m^{\ell-\epsilon}}$.
\end{theorem}
\begin{proof}
Observe that any $(\ell,r)$-box in $X$ is a sub-box of $\prod_{i=1}^{\ell}X^{k_i}$, for some $k_1,\ldots, k_{\ell}\geq 1$ with $k_1+\ldots +k_{\ell}=r$.  Given $k_1,\ldots, k_{\ell}\geq 1$ such that $k_1+\ldots +k_{\ell}=r$, let $\calF(k_1,\ldots, k_{\ell})=\calF\cap \prod_{i=1}^{\ell}X^{k_i}$. Our observation implies that $\calF$ shatters an $(\ell,r)$-box of height $m$ in $X$ if and only if $\calF(k_1,\ldots, k_{\ell})$ shatters a sub-box of $\prod_{i=1}^{\ell}X^{k_i}$ of height $m$, for some $k_1,\ldots, k_{\ell}\geq 1$ with $k_1+\ldots +k_{\ell}=r$. Consequently,
\begin{align}
\pi_{\ell}(\calF,m)&=\max\{ \pi_{\ell}(\calF(k_1,\ldots, k_{\ell}),m): k_1,\ldots, k_{\ell}\geq 1, k_1+\ldots +k_{\ell}=r\}\text{ and }\label{1}\\
\VC_{\ell}(\calF)&=\max\{ \VC_{\ell}(\calF(k_1,\ldots, k_{\ell})): k_1,\ldots, k_{\ell}\geq 1, k_1+\ldots +k_{\ell}=r\},\label{2}
\end{align}
where the left-hand sides are computed as in Definition \ref{vcdef2} and the right-hand sides are computed as in Definition \ref{vcdef1}.  By assumption, $\VC_{\ell}(\calF)\leq d$, so (\ref{2}) implies that for all $k_1,\ldots, k_{\ell}\geq 1$ with $k_1+\ldots +k_{\ell}=r$, $\VC_{\ell}(\calF(k_1,\ldots, k_{\ell}))\leq d$.  Therefore, by Theorem \ref{CPT}, there are $C=C(d)$ and $\epsilon=\epsilon(d)>0$ such that for all $m$, $\pi_{\ell}(\calF(k_1,\ldots, k_{\ell}),m)\leq C 2^{m^{\ell-\epsilon}}$.  Combining this with (\ref{1}) implies $\pi_{\ell}(\calF,m)\leq C 2^{m^{\ell-\epsilon}}$ holds for all $m$.
\end{proof}

Note that $\VC_{1}$-dimension is the same as $\VC$-dimension.  Observe that in the notation of Definition \ref{vcdef2}, for all $m$, $\pi_{\ell}(\calF,m)\leq 2^{m^{\ell}}$, and $\VC_{\ell}(\calF)\geq m$ if and only if $\pi_{\ell}(\calF,m)=2^{m^{\ell}}$.  We will be particularly interested in the $\VC_{\ell}$-dimension of families of sets defined by formulas in an $\calL$-structure. Given a formula $\varphi(\xbar; \ybar)$, an $\mathcal{L}$-structure $\calM$, and $\bbar \in M^{|\xbar|}$, let 
$$
\varphi(\bbar;\calM)=\{\abar\in M^{|\ybar|}:\calM\models \varphi(\bbar; \abar)\}\qquad \hbox{ and }\qquad \calF_{\varphi}(\calM)=\{\varphi(\bbar; \calM): \bbar \in M^{|\xbar|}\}.
$$
Note $\calF_{\varphi}(\calM)\subseteq \calP(M^{|\ybar|})$.  If $A\subseteq M^{|\ybar|}$, we say $\varphi$ \emph{shatters $A$} if $\calF_{\varphi}(\calM)$ does. Given $1\leq \ell \leq |\ybar|$, set $\VC_{\ell}(\varphi,\calM)=\VC_{\ell}(\calF_{\varphi}(\calM))$. Then if $\calH$ is a hereditary $\calL$-property, set
$$
\VC_{\ell}(\varphi,\calH)=\sup\{\VC_{\ell}(\varphi, \calM): \calM\in \calH\}.
$$
\noindent We now define the $\VC_{\ell}$-dimension of a hereditary $\calL$-property, for $\ell \geq 1$.
\begin{definition}\label{ldep}
Suppose $\ell\geq 1$, and $\calH$ is a hereditary $\calL$-property.  Then 
$$
\VC_{\ell}(\calH)=\sup\{\VC_{\ell}(\varphi,\calH): \varphi(\xbar;\ybar)\in \calL\text{ is quantifier-free}\},
$$
and we say $\calH$ is \emph{$\ell$-dependent} if for all quantifier-free formulas $\varphi(\xbar;\ybar)$, $\VC_{\ell}(\varphi,\calH)<\omega$.
\end{definition}

Note that in Definition \ref{ldep}, we define $\VC_{\ell}(\calH)$ in terms of $\VC_{\ell}(\varphi, \calH)$ for \emph{quantifier-free} $\varphi$.  Because we are dealing with classes of finite structures, this turns out to be the appropriate notion.  We now explain how this is related to the $\VC_{\ell}$-dimension of a complete first-order theory and the notion of $\ell$-dependence.  Suppose $T$ is a complete $\calL$-theory.  Given a formula, $\varphi(\xbar;\ybar)$, the \emph{$\VC_{\ell}$-dimension of $\varphi$ in $T$} is $\VC_{\ell}(\varphi,T):=\VC_{\ell}(\varphi, \calM)$, where $\calM$ is a monster model of $T$ and $\VC_{\ell}(\varphi,\calM)$ is computed precisely as described above.  The theory $T$ is \emph{$\ell$-dependent} if $\VC_{\ell}(\varphi,T)<\omega$ for all $\varphi\in \calL$. This can be related to Definition \ref{ldep} as follows.  Let $\calH(T)$ be the age of $\calM$ (i.e. the class of finite $\calL$-structures which embed into $\calM$).  Then for any quantifier-free $\varphi$, $\VC_{\ell}(\varphi,\calH(T))=\VC_{\ell}(\varphi,T)$.  Clearly if $T$ is $\ell$-dependent, then so is $\calH(T)$.  However, the converse will not hold if all quantifier-free formulas have finite $\VC_{\ell}$-dimension in $T$, but there is a $\varphi$ with quantifiers such that $\VC_{\ell}(\varphi,T)=\omega$.  Further, many hereditary $\calL$-properties are not ages (recall that if $\calL$ is finite and relational, then a hereditary $\calL$-property is an age if and only if it has the joint embedding property \cite{CAMERON199149}).  Thus, while one can view Definition \ref{ldep} as a version of $\ell$-dependence adapted to the setting of hereditary $\calL$-properties, it differs in fundamental ways from the notion of the $\VC_{\ell}$-dimension of a complete theory.

\subsection{$\VC_{\ell}^*$-dimension}

In this subsection we define $\VC_{\ell}^*$-dimension, a dual version of $\VC_{\ell}$-dimension.  This is necessary because directly generalizing $\VC_{\ell}$-dimension to the case when $\ell=0$ does not give us a useful notion.  Indeed, for any formula $\varphi(\xbar)$ and $\calL$-structure $\calM$, $\varphi$ trivially shatters a $(0,0)$-box (i.e. the empty set).  We would like to point out that $\VC_{\ell}^*$-dimension is stronger than the dual version of $\VC_{\ell}$-dimension appearing in \cite{CPT}.

We now fix some notation.  Suppose $\varphi(\xbar; \ybar)$ is a formula, $X$ is a set, and $\mathbb{A}\subseteq X^{|\ybar|}$.  A \emph{$\varphi$-type over $\mathbb{A}$ in the variables $\xbar$} is a maximal consistent subset of $\{\varphi(\xbar; \abar)^i: \abar \in \mathbb{A}, i\in \{0,1\}\}$ (where $\varphi^0=\varphi$ and $\varphi^1=\neg \varphi$).  Given an integer $n$, $S_n^{\emptyset}(\mathbb{A})$ is the set of complete types in the language of equality, using $n$ variables, and with parameters in $\mathbb{A}$.  Given $p$ in $S_{\varphi}(\mathbb{A})$ or $S_{n}^{\emptyset}(\mathbb{A})$, we say $p$ is \emph{realized} in an $\calL$-structure $\calM$ if $\mathbb{A}\subseteq M^{|\ybar|}$, and there is $\abar\in M^{|\xbar|}$ such that $\calM\models p(\abar)$.  If $\calH$ is a hereditary $\calL$-property, $S^{\calH}_{\varphi}(\mathbb{A})$ is the set of complete $\varphi$-types over $\mathbb{A}$ which are realized in some $\calM\in \calH$. 

\begin{definition}\label{def}
Suppose $\calH$ is a hereditary $\calL$-property, $m\geq 1$, $\varphi(\xbar; \ybar)$ is a formula, $X$ is a set, and $\mathbb{A}\subseteq X^{|\ybar|}$. Then $S^{\calH}_{\varphi,m}(\mathbb{A})$ is the set of all $\varphi$-types of the form $p_1(\xbar_1)\cup \ldots \cup p_m(\xbar_m)$, satisfying
\begin{enumerate}
\item For each $i\in [m]$, $p_i(\xbar_i)\in S^{\calH}_{\varphi}(\mathbb{A})$, and
\item There is $\calM\in \calH$ and \emph{pairwise distinct} $\abar_1,\ldots, \abar_m\in M^{|\xbar|}$ such that $\calM\models p_1(\abar_1)\cup \ldots \cup p_m(\abar_m)$.
\end{enumerate}
Given $\rho\in S_{2|\xbar|}^{\emptyset}(\mathbb{A})$, $S^{\calH}_{\varphi,m}(\mathbb{A},\rho)$ is the set of $p_1(\xbar_1)\cup \ldots \cup p_m(\xbar_m)\in S^{\calH}_{\varphi,m}(\mathbb{A})$ such that there is $\calM\in \calH$ and \emph{pairwise distinct} $\abar_1,\ldots, \abar_m\in M^{|\xbar|}$ with $\calM\models p_1(\abar_1)\cup \ldots \cup p_m(\abar_m)\cup \bigcup_{1\leq i\neq j\leq m}\rho(\abar_i,\abar_j)$.
\end{definition}
\noindent Observe that in the notation of Definition \ref{def}, for any $(\ell,|\ybar|)$-box $\mathbb{A}$ of height $m$, $|S^{\calH}_{\varphi}(\mathbb{A})|\leq 2^{m^{\ell}}$ and for all $\rho\in S^{\emptyset}_{2|\xbar|}(\mathbb{A})$, $|S^{\calH}_{\varphi, m}(\mathbb{A},\rho)|\leq |S^{\calH}_{\varphi}(\mathbb{A})|^{m}$. Consequently, $|S^{\calH}_{\varphi,m}(\mathbb{A},\rho)|\leq 2^{m^{\ell+1}}$.  We are now ready to define the $\VC_{\ell}^*$-dimension of a hereditary $\calL$-property, for $\ell\geq 0$.
  
\begin{definition}
Suppose $\varphi(\xbar; \ybar)$ is a formula, $\calH$ is a hereditary $\calL$-property, and $0\leq \ell \leq |\ybar|$.  Then
\begin{align*}
\VC^*_{\ell}(\varphi, \calH)&=\sup \{m\in \mathbb{N}:\text{for some $(\ell,|\ybar|)$-box $\mathbb{A}$ of height $m$ and $\rho\in S^{\emptyset}_{2|\xbar|}(\mathbb{A})$, $|S^{\calH}_{\varphi,m}(\mathbb{A},\rho)|=2^{m^{\ell+1}}\}$},
\end{align*}
and $\VC^*_{\ell}(\calH)=\sup\{\VC^*_{\ell}(\varphi,\calH): \varphi(\xbar;\ybar)\in \calL \text{ is quantifier-free}\}$.
\end{definition}

Throughout we will use the notation $\VC_{\ell}^*(\calH)=\infty$ instead of $\VC_{\ell}^*(\calH)=\omega$ (and similarly for other dimensions). We will frequently use the following observation.

\begin{observation}\label{keyob}
For all $\ell\geq 0$ and formulas $\varphi(\xbar;\ybar)$, $\VC^*_{\ell}(\varphi, \calH)\geq m$ if and only if there is an $(\ell,|\ybar|)$-box $\mathbb{A}$ of height $m$ and $\rho\in S_{2|\xbar|}^{\emptyset}(\mathbb{A})$ such that $|S^{\calH}_{\varphi}( \mathbb{A})|= 2^{m^{\ell}}$ and for all $(p_1,\ldots, p_m)$ in $S^{\calH}_{\varphi}(\mathbb{A})^m$, $p_1(\xbar_1)\cup \ldots \cup p_m(\xbar_m)\in S^{\calH}_{\varphi,m}( \mathbb{A},\rho)$.
\end{observation}

On the other hand, note that for all $\ell>0$ and formulas $\varphi(\xbar;\ybar)$, $\VC_{\ell}(\varphi, \calH)\geq m$ if and only if there is an $(\ell,|\ybar|)$-box $\mathbb{A}$ of height $m$ such that $|S^{\calH}_{\varphi}( \mathbb{A})|= 2^{m^{\ell}}$.  Therefore $\VC_{\ell}^*(\varphi,\calH)\geq m$ is a stronger statement than $\VC_{\ell}(\varphi,\calH)\geq m$. 

We now make a few remarks on our choice of definitions.  We defined $\VC_{\ell}^*$-dimension using $S^{\calH}_{\varphi,m}(\mathbb{A},\rho)$ for $\rho \in S_{2|\xbar|}^{\emptyset}(\mathbb{A})$ in order to avoid pathologies in the case when $\ell=0$.  In particular, for any non-trivial hereditary $\calL$-property $\calH$ with $\calH_{2n}\neq \emptyset$, $|S^{\calH}_{x=y,n}(\emptyset)|=2^n$.  Indeed, $\calH_{2n}\neq \emptyset$ implies that for any $\sigma\in \{0,1\}^n$, $S^{\calH}_{x=y,n}(\emptyset)$ contains $\{(x_i=y_i)^{\sigma(i)}: 1\leq i\leq n\}$.  Therefore if we defined $\VC_0^*$-dimension using $|S^{\calH}_{\varphi,m}(\mathbb{A})|$ instead of $|S^{\calH}_{\varphi,m}(\mathbb{A},\rho)|$ for some $\rho\in S_{2|\xbar|}^{\emptyset}(\mathbb{A})$, every hereditary $\calL$-property of interest to us would satisfy $\VC_0^*(x_1=x_2,\calH)=\infty$.  Our definition avoids this undesirable behavior when $\ell=0$.  Further, we will prove in Section \ref{eqsection} that for any hereditary $\calL$-property $\calH$ and $\ell>0$, $\VC_{\ell}^*(\calH)=\infty$ if and only if $\VC_{\ell}(\calH)=\infty$.  In light of this, we may extend Definition \ref{ldep} to all $\ell\geq 0$ by saying a hereditary $\calL$-property $\calH$ is \emph{$\ell$-dependent} if $\VC^*_{\ell}(\varphi,\calH)<\infty$ for all quantifier-free $\varphi$.

\section{Technical Lemmas}\label{tech}

In this section we present two technical lemmas which we will use in the proofs of our main results. Since we are interested in counting, it is often important to distinguish between tuples and their underlying sets.   For this reason we will often denote sets of tuples using bold face letters, and the corresponding underlying sets using non-bold letters. Objects which are tuples will always have bars over them. {\bf For the rest of the paper $\calL$ is a fixed finite relational language with maximum arity $r\geq 1$, and $\calH$ is a hereditary $\calL$-property.}   For the rest of the paper, ``formula'' always means quantifier-free formula.  Since $\calH$ is now fixed, we will from here on omit the superscripts $\calH$ from the notation defined in Definition \ref{def}.   

 The first result of this section is Lemma \ref{indlem} below.  Parts (a) and (b) of Lemma \ref{indlem} give quantitative bounds for the size of indiscernible sets in the language of equality, and part (c) of Lemma \ref{indlem} is an easy but useful counting fact.  The proof of Lemma \ref{indlem} is straightforward and appears in the appendix.

\begin{lemma}\label{indlem}
Suppose $X$ is a set, $s,t\in \mathbb{N}$, $\mathbb{B}\subseteq X^t$ is finite, and $B$ is the underlying set of $\mathbb{B}$.  Then the following hold.
\begin{enumerate}[(a)]
\item $|S_s^{\emptyset}(\mathbb{B})|\leq 2^{s\choose 2}(|B|+1)^s$.
\item There is $\mathbb{B}'\subseteq \mathbb{B}$ which is an indiscernible subset of $X^t$ in the language of equality satisfying $|\mathbb{B}'|\geq \Big(|\mathbb{B}|/2^{t\choose 2}\Big)^{1/2^t}$.
\item $|B|\leq t|\mathbb{B}|$ and $|\mathbb{B}|^{1/t}\leq |B|$.
\end{enumerate}
\end{lemma}

If $0<\ell\leq r$ and $\mathbb{A}=\prod_{i=1}^{\ell}A_i$ is an $(\ell,r)$-box, then a \emph{sub-box of $\mathbb{A}$} is an $(\ell,r)$-box of the form $\prod_{i=1}^{\ell}A_i'$ where for each $1\leq i\leq \ell$, $A'_i\subseteq A_i$ is nonempty.  By convention, for any $r\geq 0$, the only sub-box of a $(0,r)$-box is itself. Our next result of this section is Lemma \ref{shlem} below, which gives us information about types over sub-boxes.

\begin{lemma}\label{shlem}
Let $\varphi(\xbar; \ybar)$ be a formula, and let $\ell, K, N,m$ be integers satisfying $K>>N\geq m\geq 1$, and $0\leq \ell$.  If $\mathbb{A}$ is an $(\ell,|\ybar|)$-box of height $K$ satisfying $|S_{\varphi}(\mathbb{A})|=2^{K^{\ell}}$, then for any sub-box $\mathbb{A}'\subseteq \mathbb{A}$ of height $m$, the following hold.
\begin{enumerate}[(a)]
\item The underlying set of $\mathbb{A}'$ has size at most $|\ybar|m$.\label{a'}
\item $|S_{\varphi}(\mathbb{A}')|=2^{m^{\ell}}$.\label{b'}
\item Suppose $\ell>0$, $\calM$ is an $\calL$-structure, and $\mathbb{D}\subseteq M^{|\xbar|}$ contains one realization of every element of $S_{\varphi}(\mathbb{A})$.  Then $\mathbb{D}$ contains at least $N$ realizations of every element of $S_{\varphi}(\mathbb{A}')$.\label{c'}
\item If $|S_{\varphi,K}(\mathbb{A},\rho)|= 2^{K^{\ell+1}}$ for some $\rho \in S^{\emptyset}_{2|\xbar|}(\mathbb{A})$, then $|S_{\varphi,m}(\mathbb{A}',\rho{\upharpoonright}_{\mathbb{A}'})|=2^{m^{\ell+1}}$, and there is $\calM\in \calH$ and $\mathbb{D}\subseteq M^{m|\xbar|}$ such that $\mathbb{D}$ contains one realization of every element of $S_{\varphi,m}(\mathbb{A}',\rho{\upharpoonright}_{\mathbb{A}'})$, and $M$ contains at least $N$ elements not in $A'$ or in any element of $\mathbb{D}$.\label{d'}
\end{enumerate}
\end{lemma}
\begin{proof}
Let $A$ be the underlying set of $\mathbb{A}$ and let $A'$ be the underlying set of $\mathbb{A}'$.  We first show (\ref{a'}). If $\ell=0$, then $\mathbb{A}=\mathbb{A}'$ implies either $|\ybar|=0$ and $|A'|=0\leq |\ybar|m$, or $|\ybar|>0$ and $|A'|=1\leq |\ybar|m$.  If $\ell>0$, then $\mathbb{A}=\prod_{i=1}^{\ell}A_i$ where for each $i$, $A_i\subseteq A^{k_i}$ for some $k_i\geq 1$ and such that $\sum_{i=1}^{\ell}k_i=|\ybar|$.  Because $\mathbb{A}'$ is a sub-box of $\mathbb{A}$ of height $m$, we have $\mathbb{A}'=\prod_{i=1}^{\ell}A_i'$, where for each $i$, $A'_i\subseteq A_i$ has size $m$.  For each $i$, Lemma \ref{indlem} part (c) implies $|A'_i|\leq k_im$.  Consequently, $|A'|\leq \sum_{i=1}^{\ell}k_im=|\ybar|m$.  Thus (\ref{a'}) holds.  For parts (\ref{b'}), (\ref{c'}), and (\ref{d'}), we will use the following claim.
 \begin{claim}\label{claim}
There are $\Gamma_1,\ldots, \Gamma_{2^{m^{\ell}}}\subseteq S_{\varphi}(\mathbb{A})$ and pairwise distinct $p_1,\ldots, p_{2^{m^{\ell}}}$ in $S_{\varphi}(\mathbb{A}')$ such that:
\begin{enumerate}[(i)]
\item For each $1\leq i\leq 2^{m^{\ell}}$, every element of $\Gamma_i$ is an extension of $p_i$ to $\mathbb{A}$.
\item If $\ell>0$, then $|\Gamma_i|\geq N$.
\end{enumerate}
\end{claim}
\noindent{\it Proof.} Suppose first $\ell=0$. Then $\mathbb{A}'=\mathbb{A}$.  By assumption $|S_{\varphi}(\mathbb{A})|=2^{K^{\ell}}=2$. Let $p_1,p_2$ be the two distinct elements of $S_{\varphi}(\mathbb{A})$, and set $\Gamma_1=\{p_1\}$, and $\Gamma_2=\{p_2\}$.   Then it is clear $p_1\neq p_2\in S_{\varphi}(\mathbb{A}')$ and (i), (ii) hold.  Suppose now $\ell\geq 1$.  Let $X_1,\ldots, X_{2^{m^{\ell}}}$ enumerate all the subsets of $\mathbb{A}'$, and for each $1\leq j\leq 2^{m^{\ell}}$, set $p_j(\xbar)=  \{\varphi(\xbar;\abar): \abar\in X_j\}\cup \{\neg \varphi(\xbar;\abar):\abar\in \mathbb{A}'\setminus X\}$.  Given $X\subseteq \mathbb{A}\setminus \mathbb{A}'$, and $1\leq j\leq 2^{m^{\ell}}$, set 
$$
p_{j,X}(\xbar)=\{\varphi(\xbar; \abar): \abar \in X_j\cup X\}\cup \{\neg \varphi(\xbar; \abar): \abar \in\mathbb{A}\setminus (X_j\cup X)\}\quad \hbox{ and }\quad \Gamma_j=\{p_{j,X}(\xbar): X\subseteq \mathbb{A}\setminus \mathbb{A}'\}.
$$
Since $|S_{\varphi}(\mathbb{A})|=2^{K^{\ell}}$, we must have that for all $X\subseteq \mathbb{A}$, $\{\varphi(\xbar;\abar): \abar\in X\}\cup \{\neg \varphi(\xbar;\abar):\abar\in \mathbb{A}\setminus X\}$ is in $S_{\varphi}(\mathbb{A})$.  Consequently, for each $1\leq j\leq 2^{m^{\ell}}$, $\Gamma_j\subseteq S_{\varphi}(\mathbb{A})$.  By definition, for all $p\in \Gamma_j$, $p{\upharpoonright}_{\mathbb{A}'}=p_j$.  For each $j$, since $\Gamma_j\subseteq S_{\varphi}(\mathbb{A})$, and since any realization of an element of $\Gamma_j$ is a realization of $p_j$, we have $p_j\in S_{\varphi}(\mathbb{A}')$.  By definition, $p_1,\ldots, p_{2^{m^{\ell}}}$ are pairwise distinct.  Thus we have shown $p_1,\ldots, p_{2^{m^{\ell}}}$ are pairwise distinct elements of $S_{\varphi}(\mathbb{A}')$ and (i) holds.  For each $j$, $|\Gamma_j|\geq |\calP(\mathbb{A}\setminus \mathbb{A}')|\geq 2^{K^{\ell}}-2^{m^{\ell}}\geq N$, where the last inequality is because $K>>N\geq m$.  Thus (ii) holds.  This finishes the proof of Claim \ref{claim}.{\leavevmode\unskip\penalty9999 \hbox{}\nobreak\hfill\quad\hbox{$\blacksquare$}}

Now fix $\Gamma_1,\ldots, \Gamma_{2^{m^{\ell}}}, p_1,\ldots, p_{2^{m^{\ell}}}$ as in Claim \ref{claim}.  Since the $p_i$ are pairwise distinct elements of $S_{\varphi}(\mathbb{A}')$, we immediately have that $|S_{\varphi}(\mathbb{A}')|=2^{m^{\ell}}$, so (\ref{b'}) holds.  We now show (\ref{c'}) holds.  Suppose $\ell>0$, $\calM$ is an $\calL$-structure, and $\mathbb{D}\subseteq M^{|\xbar|}$ contains one realization of every element of $S_{\varphi}(\mathbb{A})$.  Then $\calM$ contains a realization of every element in $\bigcup_{i=1}^{2^{m^{\ell}}}\Gamma_i$.  Since each $\Gamma_i$ contains at least $N$ extensions of $p_i$, this shows $\calM$ contains at least $N$ realizations of each $p_i$.  This finishes the proof of (\ref{c'}).

We now prove (\ref{d'}).  Suppose $|S_{\varphi,K}(\mathbb{A},\rho)|=2^{K^{\ell+1}}$ for some $\rho \in S^{\emptyset}_{2|\xbar|}(\mathbb{A})$.  Since $K>>m$, we may assume $K\geq m2^{m^{\ell}+1}$.  Thus we may fix a sequence $(\alpha_{1},\ldots, \alpha_{K})\in S_{\varphi}(\mathbb{A})^K$ such that for each $1\leq j\leq 2^{m^{\ell}}$, 
\begin{align}\label{label}
|\{\alpha_1,\ldots, \alpha_{m2^{m^{\ell}}}\}\cap \Gamma_j|=m.
\end{align}
Then $|S_{\varphi,K}(\mathbb{A},\rho)|= 2^{K^{\ell+1}}$ implies by Observation \ref{keyob} that $\overline{\alpha}:=\alpha_{1}(\xbar_1)\cup \ldots \cup \alpha_{K}(\xbar_K)\in S_{\varphi,K}(\mathbb{A},\rho)$.  Thus there is $\calM\in \calH$ containing pairwise distinct $\abar_{1},\ldots, \abar_{K}$ realizing $\overline{\alpha}$ such that for each $i\neq j$, $\calM\models \rho(\abar_i,\abar_j)$.  For each $1\leq j\leq 2^{m^{\ell}}$, since every element of $\Gamma_j$ extends $p_j$, (\ref{label}) implies that $\{\abar_1,\ldots, \abar_{m2^{m^{\ell}}}\}$ contains $m$ realizations of $p_j$.  This means that for all $(p_{j_1},\ldots, p_{j_m})$ in $S_{\varphi}(\mathbb{A}')^m$, we may choose pairwise distinct tuples $\cbar_{1},\ldots, \cbar_{m}\in \{\abar_1,\ldots, \abar_{m2^{m^{\ell}}}\}$ with the property that $\calM\models p_{j_1}(\cbar_{1})\cup \ldots\cup p_{j_m}(\cbar_{m})$.  Let $\mathbb{D}$ consist of one such realization for each $(p_{j_1},\ldots, p_{j_m})\in S_{\varphi}(\mathbb{A}')^m$.  Note that $(\cbar_1,\ldots, \cbar_m)\in \mathbb{D}$ implies $\calM\models \rho {\upharpoonright}_{\mathbb{A}'}(\cbar_i,\cbar_j)$ for each $i\neq j$ (since $\mathbb{D}\subseteq\{\abar_1,\ldots, \abar_{m2^{m^{\ell}}}\}^m$).

We have now shown that for every $(p_{j_1},\ldots, p_{j_m})\in S_{\varphi}(\mathbb{A}')^m$, $p_{j_1}(\xbar_1)\cup \ldots \cup p_{j_m}(\xbar_m)$ is in $S_{\varphi,m}(\mathbb{A'},\rho{\upharpoonright}_{\mathbb{A}'})$.  To finish the proof of (\ref{d'}), we just have to show that $M$ contains at least $N$ elements not appearing in $A'$ or $\mathbb{D}$.  Let $D$ be the underlying set of $\mathbb{D}$ and let $E$ be the underlying set of the tuples $\mathbb{E}=\{\abar_{m2^{m^{\ell}}+1}, \ldots, \abar_K\}$.  Since $K\geq m2^{m^{\ell}+1}$, $|\mathbb{E}|\geq K/2$.  This along with Lemma \ref{indlem} part (c) and the fact that $\mathbb{E}\subseteq M^{|\xbar|}$ implies $(K/2)^{1/|\xbar|}\leq |\mathbb{E}|^{1/|\xbar|}\leq |E|$.  Since $\mathbb{D}\subseteq M^{|\xbar|m}$ and $|\mathbb{D}|=m2^{m^{\ell}}$, Lemma \ref{indlem} part (c) implies $|D|\leq |\xbar|m2^{m^{\ell}}$. We have already shown $|A'|\leq |\ybar|m$.  Combining these bounds, we obtain  that  $|E\setminus (A'\cup D)|\geq (K/2)^{1/|\xbar|}- |\xbar|m2^{m^{\ell}}-|\ybar|m\geq N$, where the last inequality is because $K>>N\geq  m$. This finishes the proof of (\ref{d'}).
\end{proof}

\section{Proofs of Main Theorems}\label{secth5}

In this section, we prove the main results of this paper.  We begin with Theorem \ref{th5}, which we restate here for convenience.  If $\calM$ is an $\calL$-structure and $A\subseteq M$, then $\calM[A]$ denotes the $\calL$-structure induced on $A$ by $\calM$.

 \begin{theorem4}
 If $1\leq \ell$ and $\VC^*_{\ell-1}(\calH)=\infty$, then there is $C>0$ such that for sufficiently large $n$, $|\calH_n|\geq 2^{Cn^{\ell}}$.
\end{theorem4}
\begin{proof}
%fix this by showing we may assume $\varphi$ is proper.
Assume $1\leq \ell$ and $\VC^*_{\ell-1}(\calH)=\infty$.  By definition, there is a formula $\varphi(\xbar; \ybar)$ such that $VC^*_{\ell-1}(\varphi,\calH)=\infty$.  Let $s=|\xbar|$ and $t=|\ybar|+|\xbar|$.  Fix $n$ large and $K>>n$.  Then $\VC_{\ell-1}^*(\varphi,\calH)\geq K$ implies there is an $(\ell-1,|\ybar|)$-box $\mathbb{A}$ of height $K$ and $\rho\in S_{2s}^{\emptyset}(\mathbb{A})$ such that $|S_{\varphi,K}(\mathbb{A},\rho)|=2^{K^{\ell}}$.  Fix $m=\lfloor n/3t\rfloor$.  Note $m\leq n<<K$.

Choose a sub-box $\mathbb{A}'\subseteq \mathbb{A}$ of height $m$ and let $A'$ be the underlying set of $\mathbb{A}'$.  By Lemma \ref{shlem} part (\ref{a'}), $|A'|\leq |\ybar| m\leq tk$.  By Lemma \ref{shlem} part (\ref{d'}), $|S_{\varphi,m}(\mathbb{A}',\rho{\upharpoonright}_{\mathbb{A}'})|=2^{m^{\ell}}$, and there is $\calM\in \calH$, and $\mathbb{D}\subseteq M^{m|\xbar|}$ such that $\mathbb{D}$ contains one realization of every element of $S_{\varphi,m}(\mathbb{A}',\rho{\upharpoonright}_{\mathbb{A}'})$, and $M$ contains $n$ elements not appearing in $A'$ or in $\mathbb{D}$. Let $D$ be the underlying set of $\mathbb{D}$ and let $E\subseteq M$ be a set of $n$ elements in $M\setminus (A'\cup D)$.

Given $\Cbar\in \mathbb{D}$, let $C$ be the underlying set of $\Cbar$.  For all $\Cbar\in \mathbb{D}$, $\Cbar \in M^{m|\xbar|}$ implies by Lemma \ref{indlem} part (c) that $|C|\leq |\xbar|m\leq tm$.  Since every element of $\mathbb{D}$ realizes the same equality type over $\mathbb{A}'$, we have that for all $\Cbar, \bar{C'}\in \mathbb{D}$, $|C|=|C'|$ and $|C\cup A'|=|C\cup A'|$.  Given $\Cbar\in \mathbb{D}$, note 
\begin{align}\label{ppp}
|C\cup A'|\leq |C|+|A'|\leq tm+tm=2tm\leq 2t(n/3t)=2n/3.
\end{align}
Since $E$ has size $n$ and is disjoint from $D\cup A'$, (\ref{ppp}) implies we may choose $E'\subseteq E$ such that for all $\Cbar\in \mathbb{D}$, $|C\cup A'\cup E'|=n$.  Now for each $\Cbar\in \mathbb{D}$, set $\calM_{\Cbar}=\calM[C\cup A'\cup E']$.  Because $\calH$ is a hereditary $\calL$-property, $\calM_{\Cbar}\in \calH$ for all $\Cbar\in \mathbb{D}$.  Fix some $\Cbar_*=(\cbar^*_1,\ldots, \cbar^*_m)\in \mathbb{D}$.  Given $\Cbar=(\cbar_1,\ldots, \cbar_m)\in \mathbb{D}$, note that $\Cbar_*$ and $\Cbar$ have the same equality type over $A'\cup E'$.  Therefore there is a bijection $f_{\Cbar}:C\cup A'\cup E'\rightarrow C_*\cup A'\cup E'$, which fixes $A'\cup E'$ and which sends $\cbar_i$ to $\cbar^*_i$ for each $1\leq i\leq m$.  Let $\calM_{\Cbar}^*$ be the $\calL$-structure with universe $C_*\cup A'\cup E'$, and which is isomorphic to $\calM_{\Cbar}$ via the bijection $f_{\Cbar}$.  Since $\calH$ is closed under isomorphism, $\calM_{\Cbar}^*\in \calH$ for all $\Cbar\in \mathbb{D}$.  Clearly $\Cbar\neq \Cbar'$ implies $\calM_{\Cbar}^*\neq \calM_{\Cbar'}^*$ (since then $\Cbar$ and $\Cbar'$ realize distinct elements of $S_{\varphi,m}(\mathbb{A},\rho{\upharpoonright}_{\mathbb{A}'})$).  Thus $\{\calM_{\Cbar}^*: \Cbar \in \mathbb{D}\}$ consists of $|\mathbb{D}|$ distinct elements of $\calH$, all with universe $C_*\cup A'\cup E'$.  Since $|C_*\cup A'\cup E'|=n$ and $\calH$ is closed under isomorphism, this shows $|\calH_n|\geq |\mathbb{D}|= 2^{m^{\ell}}$.  Since $m=\lfloor n/3t\rfloor$ and $n$ is large, we have $|\calH_n|\geq 2^{m^{\ell}}\geq 2^{Cn^{\ell}}$ for $C=(1/4t)^{\ell}$. 
\end{proof}

We will use the following result from \cite{CTHParxiv} in our proof of Theorem \ref{th2}.

\begin{theorem}\label{enumeration}
Suppose $\mathcal{H}$ is a hereditary $\calL$-property.  Then the following limit exists.
$$
\pi(\mathcal{H})=\lim_{n\rightarrow \infty}|\calH_n|^{1/{n\choose r}}.
$$
Moreover, if $\pi(\mathcal{H})>1$, then $|\mathcal{H}_n|= \pi(\mathcal{H})^{{n\choose r}+o(n^r)}$, and if $\pi(\mathcal{H})\leq 1$, then $|\mathcal{H}_n|=2^{o(n^r)}$.  
\end{theorem}

We now fix some notation.  A formula $\varphi(\xbar;\ybar)$ is \emph{trivially partitioned} if $|\ybar|=0$.  Given a set $X$ and $n\geq 1$, $X^{\underline{n}}=\{(x_1,\ldots, x_n)\in X^n:i\neq j$ implies $x_i\neq x_j\}$.  If $\calM\in \calH$, $\varphi(\xbar;\ybar)$ is a formula, $\mathbb{A}\subseteq M^{|\ybar|}$, and $\abar_1,\ldots, \abar_k\in M^{|\xbar|}$ are pairwise distinct, then define $qftp^{\calM}_{\varphi}(\abar_1,\ldots, \abar_k; \mathbb{A})$ to be the element $p_1(\xbar_1)\cup \ldots \cup p_k(\xbar_k)$ of $S_{\varphi,k}(\mathbb{A})$ such that $\calM\models p_1(\abar_1)\cup \ldots \cup p_k(\abar_k)$. 

The following notation is from \cite{AroskarCummings}.  Let $Index$ be the set of pairs $(R,p)$ where $R(x_1,\ldots, x_t)$ is a relation of $\calL$ and $p$ is a partition of $[t]$.  Given $(R,p)\in Index$, define $R_p(\zbar)$ be the formula obtained as follows.  Suppose $p_1,\ldots, p_s$ are the parts of $p$, and for each $i$, $m_i=\min p_i$.  For each $x_j\in \{x_1,\ldots, x_t\}$, find which part of $p$ contains $j$, say $p_i$, then replace $x_j$ with $x_{m_i}$.  Relabel the variables $(x_{m_1},\ldots, x_{m_s})=(z_1,\ldots, z_s)$ and let $R_p(\zbar)$ be the resulting formula.  Now let $rel(\calL)$ consist of all formulas $\varphi(\ubar;\vbar)$ obtained by permuting and/or partitioning the variables of a formula of the form $R_p(\zbar)\wedge \bigwedge_{1\leq i\neq j\leq {|\zbar|}}z_i\neq z_j$, where $(R,p)\in Index$.  

Given a formula $\varphi(\ubar;\vbar)$ and an $\calL$-structure $\calM$, let $\varphi(\calM)=\{\abar\bbar\in M^{|\ubar|+|\vbar|}: \calM\models \varphi(\abar;\bbar)\}$.  Observe that if $\varphi(\ubar;\vbar)\in rel(\calL)$ and $\calM$ is an $\calL$-structure, then $\varphi(\calM)\subseteq M^{\underline{|\ubar|+|\vbar|}}$, and $|\ubar|+|\vbar|\leq r$.   We will use the fact that any $\calL$-structure $\calM$ is completely determined by knowing $\varphi(\calM)$ for all trivially partitioned $\varphi\in rel(\calL)$, or by knowing $\varphi(\calM)$ for all $\varphi(\ubar;\vbar)\in rel(\calL)$ with $|\ubar|=1$.  Given a formula $\varphi(\ubar;\vbar)$ and $n\geq 1$, set 
$$
\calF_{\varphi}(n):=\{U\subseteq [n]^{|\ubar|+|\vbar|}: \text{ there is }\calM\in \calH_n\text{ with }\varphi(\calM)=U\}.
$$ 
We now prove Theorem \ref{zerodim} and then Theorem \ref{th2}, which we restate here for convenience. Recall $\calH$ is a fixed hereditary $\calL$-property and the maximum arity of $\calL$ is $r$.

\begin{theorem3}
One of the following holds.
\begin{enumerate}[(a)]
\item $VC^*_{0}(\calH)<\infty$ and there $K>0$ such that for sufficiently large $n$, $|\calH_n|\leq n^K$ or 
\item $VC^*_{0}(\calH)=\infty$, and there is a constant $C>0$ such that for sufficiently large $n$, $|\calH_n|\geq 2^{Cn}$.
\end{enumerate}
\end{theorem3}
\begin{proof}
If $\VC_{0}^*(\calH)=\infty$, then Theorem \ref{th5} implies there is a constant $C>0$ such that for large $n$, $|\calH_n|\geq 2^{Cn}$, so (b) holds. Suppose now $VC^*_{0}(\calH)=d<\infty$. Fix $\varphi(\xbar)$ a trivially partitioned formula from $rel(\calL)$.  Set $k=(d+1)^r2^{r\choose 2}$ and fix $n>>k,d,|\xbar|$.  Observe $\calF_{\varphi}(n)\subseteq [n]^{\underline{|\xbar|}}$ because $\varphi\in rel(\calL)$.  We show $\VC(\calF_{\varphi}(n))< k$.  Suppose towards a contradiction $\VC(\calF_{\varphi}(n))\geq k$.  Then there is $U\subseteq [n]^{\underline{|\xbar|}}$ of size $k$ shattered by $\calF_{\varphi}(n)$.  In other words, for all $Y\subseteq U$, there is $\calM_Y\in \calH_n$ with $\varphi(\calM_Y)=Y$.  Lemma \ref{indlem} part (b) implies there is $U^*\subseteq U$ which is an indiscernible set with respect to equality, and which has size at least $(k/2^{|\xbar|\choose 2})^{1/|\xbar|}\geq (k/2^{r\choose 2})^{1/r}=d+1$. Let $V=\{\vbar_1,\ldots, \vbar_{d+1}\}$ consist of $d+1$ distinct elements of $U^*$.  Let $\rho\in S_{2|\xbar|}^{\emptyset}(\emptyset)$ be such that for all $i\neq j$, $\rho(\vbar_i,\vbar_j)$ holds (this exists because $V\subseteq U^*$ and $U^*$ is an indiscernible set with respect to equality).  Note that for any $Y,Y'\subseteq V$, $Y\neq Y'$ implies $qftp^{\calM_Y}_{\varphi}(\vbar_1,\ldots, \vbar_{d+1})\neq qftp^{\calM_{Y'}}_{\varphi}(\vbar_1,\ldots, \vbar_{d+1})$. This shows $|S_{\varphi,d+1}(\emptyset,\rho)|=2^{|V|}=2^{d+1}$, contradicting that $\VC_0^*(\calH)=d$.  Thus $|\VC(\calF_{\varphi}(n))|\leq k$, and consequently, $|\calF_{\varphi}(n)|\leq Cn^k$, where $C=C(k)>0$ is from Theorem \ref{ssl}.  Every $\calM\in \calH_n$ can be built by choosing, for each trivially partitioned $\varphi(\xbar)\in rel(\calL)$, an element of $\calF_{\varphi}(n)$ to be $\varphi(\calM)$. Hence
$$
|\calH_n|\leq  \prod_{\varphi\in rel(\calL)}|\calF_{\varphi}(n)|\leq (Cn^k)^{|rel(\calL)|}=C^{|rel(\calL)|}n^{|rel(\calL)|k} \leq n^{2|rel(\calL)|k},
$$
where the last inequality is because $n$ is large and $|rel(\calL)|, C$ are constants. Thus (a) holds where $K=2|rel(\calL)|k$.
\end{proof}

\begin{theorem2}
One of the following holds.
\begin{enumerate}[(a)]
\item $VC^*_{r-1}(\calH)<\infty$ and there is an $\epsilon>0$ such that for sufficiently large $n$, $|\calH_n|\leq 2^{n^{r-\epsilon}}$ or
\item $VC^*_{r-1}(\calH)=\infty$, and there is a constant $C>0$ such that $|\calH_n|= 2^{Cn^{r}+o(n^r)}$.
\end{enumerate}
When $r=1$, (a) can be replaced by the following stronger statement: 
\begin{enumerate}[(a')]
\item $VC^*_{0}(\calH)<\infty$ and there is a constant $K>0$ such that for sufficiently large $n$, $|\calH_n|\leq n^K$.
\end{enumerate}
\end{theorem2}
\begin{proof}
If $\VC_{r-1}^*(\calH)=\infty$, then Theorem \ref{th5} implies there is a constant $C$ such that for large $n$, $|\calH_n|\geq 2^{Cn^r}$.  By Theorem \ref{enumeration}, $\pi(\calH)>1$ and $|\calH_n|=\pi(\calH)^{{n\choose r}+o(n^r)}$.  Clearly this implies there is $C'>0$ such that $|\calH_n|=2^{C'n^r+o(n^r)}$, so we have shown (b) holds.

Assume now $\VC_{r-1}^*(\calH)=d<\infty$.  If $r=1$, then (a') holds by Theorem \ref{zerodim}. So assume $r\geq 2$.  Fix $\varphi(\xbar;\ybar)\in rel(\calL)$ with $|\xbar|=1$ and $n>>d$.  Observe $\calF_{\varphi}(n)\subseteq [n]^{\underline{1+|\ybar|}}$ because $\varphi\in rel(\calL)$.  We show $\VC_{r}(\calF_{\varphi}(n))\leq d$.  If $1+|\ybar|<r$, this is obvious from the definition, so assume $1+|\ybar|=r$.  Suppose towards a contradiction $\VC_r(\calF_{\varphi}(n))>d$.  Then there is an $(r,r)$-box $\mathbb{A}\subseteq [n]^{\underline{r}}$ of height $d+1$ such that $\calF_{\varphi}(n)$ shatters $\mathbb{A}$.  In other words, if $U_1,\ldots, U_{2^{(d+1)^r}}$ enumerate the subsets of $\mathbb{A}$, then for each $1\leq j\leq 2^{(d+1)^r}$, there is $\calM_j\in \calH_n$ with $\varphi(\calM_j)=U_j$.  By definition, $\mathbb{A}=\prod_{i=1}^rA_i$, for some $A_1,\ldots, A_r\subseteq [n]$.  Enumerate $A_1=\{a_1,\ldots, a_{d+1}\}$, and set $\mathbb{A}'=\prod_{i=2}^rA_i$.  Let $\rho\in S_{d+1}^{\emptyset}(\mathbb{A}')$ be such that $\rho(a_1,\ldots, a_{d+1})$ holds.  Since $\mathbb{A}\subseteq [n]^{\underline{r}}$, $\rho(x_1,\ldots, x_{d+1})$ says all the $x_i$ are pairwise distinct and are distinct from all the elements in $\mathbb{A}'$.  Note that for each $1\leq i\neq j\leq 2^{(d+1)^r}$, $qftp_{\varphi}^{\calM_i}(a_1,\ldots, a_{d+1}; \mathbb{A}')\neq qftp_{\varphi}^{\calM_j}(a_1,\ldots, a_{d+1}; \mathbb{A}')$ are distinct elements of $S_{\varphi,d+1}(\mathbb{A}',\rho)$.  This implies $|S_{\varphi,d+1}(\mathbb{A}',\rho)|=2^{(d+1)^r}$.  But now $\VC_{r-1}^*(\varphi,\calH)\geq d+1$, contradicting our assumption that $\VC^*_{r-1}(\calH)=d$.  Thus $\VC_r(\calF_{\varphi}(n))\leq d$.  Consequently, $|\calF_{\varphi}(n)|\leq C2^{n^{r-\epsilon}}$, where $C=C(d)$ and $\epsilon=\epsilon(d)>0$ are from Theorem \ref{CPT'}.  Every $\calM\in \calH_n$ can be built by choosing, for each  $\varphi(\xbar;\ybar)\in rel(\calL)$ with $|\xbar|=1$, an element of $\calF_{\varphi}(n)$ to be $\varphi(\calM)$. Thus 
$$
|\calH_n|\leq  \prod_{\varphi\in rel(\calL)}|\calF_{\varphi}(n)|\leq(C2^{n^{r-\epsilon}})^{|rel(\calL)|}=C^{|rel(\calL)|}2^{|rel(\calL)|n^{r-\epsilon}} \leq 2^{n^{r-\epsilon/2}},
$$
where the last inequality is because $n$ is large and $|rel(\calL)|, C$ are constants. Thus (a) holds.

\end{proof}

We end this section with Example \ref{ex1}, which shows that $\VC_{\ell}(\calH)<\infty$ does not necessarily imply $|\calH_n|\leq 2^{o(n^{\ell+1})}$, when $0<\ell<r-1$. In particular, we give an example of a hereditary $\calL$-property $\calH$ where the largest arity of $\calL$ is $3$, where $\VC_1(\calH)<\infty$, but where $|\calH_n|\geq 2^{Cn^2}$, for some $C>0$.

\begin{example}\label{ex1}
A \emph{$3$-uniform hypergraph} is a pair $(V,E)$ where $V$ is a set of vertices and $E\subseteq {V\choose 3}$.  A \emph{sub-hypergraph} of $(V,E)$ is a pair $(V,E')$ where $E'\subseteq E$.  Given a $3$-uniform hypergraph $G=(V,E)$ and $xy\in {V\choose 2}$, let $d^G(xy)=|\{e\in E: xy\subseteq e\}|$.  Let $\calL=\{E(x,y,z)\}$ and let $\calH$ be the hereditary $\calL$-property consisting of finite $3$-uniform hypergraphs $G=(V,E)$ with the property that for all pairs $xy\in {V\choose 2}$, $d^G(xy)\leq 1$.  It is straightforward to verify that $\VC(\calH)=1<\infty$.   

A \emph{Steiner triple system} is a $3$-uniform hypergraph $G=(V,E)$ with the property that for all $xy\in {V\choose 2}$, $d^G(xy)= 1$.  By \cite{steinersystem1, steinersystem2}, if $n\equiv 1\mod 6$ or $n\equiv 3 \mod 6$, then there exists a Steiner triple system on $n$ vertices. For all $n$ satisfying $n\equiv 1\mod 6$ or $n\equiv 3 \mod 6$, let $G_n$ be a Steiner triple system with vertex set $[n]$.  Then if $n$ is large, $e(G_n)=\frac{{n\choose 2}}{{3\choose 2}}\geq \frac{n^2}{7}$.  Consequently, the number of sub-hypergraphs of $G_n$ is at least $2^{\frac{n^2}{7}}$.  Clearly any sub-hypergraph of $G_n$ is in $\calH_n$, so $|\calH_n|\geq 2^{\frac{n^2}{7}}$.  We now show that for all sufficiently large $n$, $|\calH_n|\geq 2^{\frac{n^2}{14}}$.  Assume $n$ is sufficiently large.  If $n\equiv 1\text{ mod }6$ or $n\equiv 3\text{ mod }6$, then we have already shown $|\calH_n|\geq 2^{\frac{n^2}{7}}>2^{\frac{n^2}{14}}$.  If $n\not\equiv 1\text{ mod }6$ and $n\not\equiv 3\text{ mod }6$, then for some $i\in \{1,2\}$, one of $n-i\equiv 1\text{ mod }6$ or $n-i\equiv 3\text{ mod }6$ holds. Note $|\calH_n|\geq |\calH_{n-i}|$ because for all $([n-i],E)\in \calH_{n-i}$, we have $([n],E)\in \calH_n$.  Thus 
$$
|\calH_{n}|\geq |\calH_{n-i}|\geq 2^{\frac{(n-i)^2}{7}}\geq 2^{\frac{(n-2)^2}{7}}=2^{\frac{n^2}{7}-\frac{4n}{7}+\frac{4}{7}}\geq 2^{\frac{n^2}{14}},
$$
where the last inequality is because $n$ is large.  Thus $\VC(\calH)=\VC_1(\calH)=1$ but $|\calH_n|\geq 2^{n^2/14}$.
\end{example}

\begin{comment}
(Can we extend to other $r$ and $\ell$?, to do so we need to understand whether there is a constant $C$ such that for all $n$, there is $Cn\leq M\leq n$ satisfying the divisibility requirements for the existence of a $(M,r,\ell)$-steiner system: for all $0\leq i\leq r-1$, ${r-i\choose \ell-i}$ divides ${M-i\choose \ell-i}$.) 
\end{comment}

\section{Equivalence of $\VC_{\ell}(\calH)=\infty$ and $\VC_{\ell}^*(\calH)=\infty$ when $\ell\geq 1$.}\label{eqsection}

In this section we prove that when $1\leq \ell $, $\VC_{\ell}(\calH)=\infty$ if and only if $\VC^*_{\ell}(\calH)=\infty$.  

\begin{theorem}\label{eqth}
For all $1\leq \ell$, $\VC_{\ell}(\calH)=\infty$ if and only if $\VC^*_{\ell}(\calH)=\infty$
\end{theorem}
\begin{proof}

Suppose $\VC^*_{\ell}(\calH)=\infty$.  Fix $d$.  We show $\VC_{\ell}(\calH)\geq d$.  Let $N>>d$ and choose $\varphi(\xbar;\ybar)$ such that $\VC_{\ell}^*(\varphi,\calH)=\infty$.  Then $\VC^*_{\ell}(\varphi,\calH)\geq N$ implies there an $(\ell,|\ybar|)$-box $\mathbb{A}=\prod_{i=1}^{\ell}A_i$ of height $N$ and $\rho\in S^{\emptyset}_{2|\xbar|}(\mathbb{A})$ such that $|S_{\varphi,N}(\mathbb{A},\rho)|= 2^{N^{\ell+1}}$.  Fix a sub-box $\mathbb{A}'$ of $\mathbb{A}$ of height $d$.  By Lemma \ref{shlem} parts (\ref{b'}) and (\ref{d'}), $|S_{\varphi}(\mathbb{A}')|=2^{d^{\ell}}$ and there is $\calM\in \calH$ realizing every element of $S_{\varphi,d}(\mathbb{A}',\rho{\upharpoonright}_{\mathbb{A}'})$. Consequently, $\calM$ realizes every element of $S_{\varphi}(\mathbb{A}')$.  Thus $\varphi$ shatters $\mathbb{A}'$ in $\calM$, and  $\VC_{\ell}(\calH)\geq d$.

Suppose conversely $\VC_{\ell}( \calH)=\infty$.  Fix $d\in \mathbb{N}$.  We show $\VC_{\ell}^*(\calH)\geq d$. Choose $\varphi(\xbar;\ybar)$ such that $\VC_{\ell}(\varphi,\calH)=\infty$.  Let $s=|\xbar|$, $t=|\ybar|$.  Fix $K>>n>>d, s,t$, and let $C=C(n)$, $\epsilon=\epsilon(n)>0$ be from Theorem \ref{CPT'}.  Note that $C=C(n)$ implies $K>>C$.  Since $\VC_{\ell}( \varphi,\calH)\geq K$, there is an $(\ell,|\ybar|)$-box $\mathbb{A}$ of height $K$ and $\calM\in \calH$ such that $\varphi(\xbar;\ybar)$ shatters $\mathbb{A}$ in $\calM$. Let $\mathbb{D}\subseteq M^{|\xbar|}$ contain one realization of each element of $S_{\varphi}(\mathbb{A})$, and let $A$ be the underlying set of $\mathbb{A}$.  Note $|\mathbb{D}|=2^{K^{\ell}}$.  By Lemma \ref{shlem} part (a), $|A|\leq Kt$. Combining this with Lemma \ref{indlem} part (a) yields that $|S^{\emptyset}_s(\mathbb{A})|\leq 2^{s\choose 2}(|A|+1)^s\leq 2^{s\choose 2}(Kt+1)^s$. Consequently, there is $\nu(\xbar)\in S_s^{\emptyset}(\mathbb{A})$ such that 
$$
|\{\abar \in \mathbb{D}: \calM\models \nu(\abar)\}|\geq |\mathbb{D}|/2^{s\choose 2}(Kt+1)^s=2^{K^{\ell}}/2^{s\choose 2}(Kt+1)^s\geq C2^{K^{\ell-\epsilon/10}},
$$
where the last inequality is because $K>>C,s,t,n$ and $\ell\geq 1$.  Let $\mathbb{D}'=\{\abar\in \mathbb{D}: \calM\models \nu(\abar)\}$.  By Lemma \ref{indlem} part (b), there is $\mathbb{D}''\subseteq \mathbb{D}'$ which is an indiscernible set in the language of equality such that 
$$
|\mathbb{D}''|\geq \Big(|\mathbb{D}'|/2^{s\choose 2}\Big)^{1/2^s}\geq \frac{C^{1/2^s}2^{K^{(\ell-\epsilon/10)}/2^s}}{2^{{s\choose 2}/2^s}}\geq C2^{K^{\ell-\epsilon/5}},
$$
 where the last inequality is because $K>>C,s,n$ and $\ell\geq 1$.  Our definition of $\mathbb{D}''$ implies there is $\rho(\xbar,\ybar)\in S_{2s}^{\emptyset}(\mathbb{A})$ such that for every $\abar\neq \bbar\in\mathbb{D}''$, $\calM\models \rho(\abar,\bbar)$.  Now let $\calF=\{\varphi(\abar; \calM)\cap \mathbb{A}: \abar\in \mathbb{D}''\}$.  Since the elements of $\mathbb{D}''$ realize distinct $\varphi$-types over $\mathbb{A}$, $|\calF|\geq |\mathbb{D}''|\geq C2^{K^{\ell-\epsilon/5}}$.  Thus by Theorem \ref{CPT'}, $\calF$ shatters a sub-box $\mathbb{A}'$ of $\mathbb{A}$ of height $n$. This implies  there is a set $\mathbb{D}'''\subseteq \mathbb{D}''$ containing one realization of every element of $S_{\varphi}(\mathbb{A}')$, and $|S_{\varphi}(\mathbb{A}')|=2^{n^{\ell}}$.  Now let $\mathbb{B}$ be a sub-box of $\mathbb{A}'$ of height $d$.  By Lemma \ref{shlem} parts (\ref{b'}) and (\ref{c'}), $|S_{\varphi}(\mathbb{B})|=2^{d^{\ell}}$, and  $\mathbb{D}'''$ contains at least $d$ realizations of every element of $S_{\varphi}(\mathbb{B})$.  This implies that for every $(p_{i_1},\ldots, p_{i_d})\in S_{\varphi}(\mathbb{B})^d$, there are pairwise distinct $\abar_{i_1},\ldots, \abar_{i_d}$ in $\mathbb{D}'''$ realizing $p_{i_1}(\xbar_1)\cup \ldots \cup p_{i_d}(\xbar_d)$.  Because $\abar_{i_1},\ldots, \abar_{i_d}$ are in $\mathbb{D}'''\subseteq \mathbb{D}''$ and $\mathbb{B}\subseteq \mathbb{A}$, we have that $\calM\models \rho{\upharpoonright}_{\mathbb{B}}(\abar_{i_u},\abar_{i_v})$ for all $1\leq u\neq v\leq d$.  Thus we have shown $|S_{\varphi,d}(\mathbb{B},\rho{\upharpoonright}_{\mathbb{B}})|\geq |S_{\varphi}(\mathbb{B})^d|=2^{d^{\ell+1}}$, and consequently, $\VC_{\ell}^*(\calH)\geq \VC_{\ell}^*(\varphi,\calH)\geq d$.
\end{proof}

\begin{comment}

\begin{example}
Let $\calL=\{E(x,y,z)\}$ and let $\calH$ be the hereditary $\calL$-property consisting of finite $3$-uniform hypergraphs $G=(V,E)$ with the property that for some partition $X,Y$ of $V$, there is a matching $E'\subseteq {Y\choose 2}$ such that every $e$ is of the form $\{x\}\cup e'$ where $e'\in E'$.  Clearly $\calH$ is a hereditary $\calL$-property and $\VC(\calH)=\infty$.  However every element of $\calH_n$ can be constructed as follows.
\begin{enumerate}
\item Choose the partition $X,Y$ of $[n]$.  There are $2^n$ choices.
\item Choose a matching $E'\subseteq {Y\choose 2}$.  There are at most $n^{n/2}$ choices \cite{find a citation}.
\item Choose edges by choosing pairs $(x,e')$ where $x\in X$ and $e'\in E'$.  There are at most $2^{|X||E'|}\leq 2^{n(n/2)}=2^{n^2/2}$.
\end{enumerate}
In total we find that $|\calH_n|\leq 2^n n^{n/2} 2^{n^2/2}\leq 2^{n^2}$, where the last inequality is because $n$ is large.  Thus $\VC(\calH)=\infty$ but $|\calH_n|<2^{n^{3-\epsilon}}$ where $\epsilon =1$.
\end{example}

\end{comment}

\section{Appendix}
In this appendix we prove Lemma \ref{indlem}.

\begin{lemma1}
Suppose $X$ is a set, $s,t\in \mathbb{N}$, $\mathbb{B}\subseteq X^t$ is finite, and $B$ is the underlying set of $\mathbb{B}$.  Then the following hold.
\begin{enumerate}[(a)]
\item $|S_s^{\emptyset}(\mathbb{B})|\leq 2^{s\choose 2}(|B|+1)^s$.
\item There is $\mathbb{B}'\subseteq \mathbb{B}$ which is an indiscernible subset of $X^t$ in the language of equality satisfying $|\mathbb{B}'|\geq \Big(|\mathbb{B}|/2^{t\choose 2}\Big)^{1/2^t}$.
\item $|B|\leq t|\mathbb{B}|$ and $|\mathbb{B}|^{1/t}\leq |B|$.
\end{enumerate}
\end{lemma1}
\begin{proof}
Every $p(x_1,\ldots, x_s)=p(\xbar)\in S_s^{\emptyset}(\mathbb{B})$ can be constructed as follows.
\begin{itemize}
\item Choose $S\subseteq {[s]\choose 2}$, and for each $ij\in S$, put $x_i=x_j$ in $p(\xbar)$ and for each $ij\notin S$, put $x_i\neq x_j$ in $p(\xbar)$.  There are at most $2^{s\choose 2}$ ways to do this.
\item For each $i\in [s]$, do one of the following. Either put $x_i\neq b$ in $p(\xbar)$ for all $b\in B$, or choose $b\in B$ and then put $x_i=b$ in $p(\xbar)$ and put $x_i\neq b'$ in $p(\xbar)$ for all $b'\in B\setminus \{b\}$.  There are at most $(|B|+1)^s$ ways to do this.
\end{itemize}
This shows $|S_s^{\emptyset}(\mathbb{B})|\leq 2^{s\choose 2}(|B|+1)^s$, so we have proved part (a).  We now prove (b).  First, by part (a), there are at most $2^{t\choose 2}$ equality types over the empty set in the variables $x_1,\ldots, x_t$, so there is $\mathbb{B}_0\subseteq \mathbb{B}$ with $|\mathbb{B}_0|\geq |\mathbb{B}|/2^{t\choose 2}$ such that all elements in $\mathbb{B}_0$ have the same equality type over the emptyset.  Let $q(\xbar)\in S^{\emptyset}_t(\emptyset)$ be such that for all $\bbar\in \mathbb{B}_0$, $q(\bbar)$ holds.  Let $B_0$ be the underlying set of $\mathbb{B}_0$.  We now build a sequence $Y_1\supseteq Y_2\supseteq \ldots \supseteq Y_t$ such that for each $1\leq i\leq t$, $|Y_i|\geq |\mathbb{B}_0|^{1/2^i}$ and $Y_t$ is an indiscernible set in the language of equality.

\underline{Step 1:} Let $B_1=\{b\in B_0: \text{ there is }(b_1,\ldots, b_t)\in \mathbb{B}_0\text{ with }b=b_1\}$.  If there is $b\in B_1$ such that $|\{(b_1,\ldots, b_t)\in \mathbb{B}_0: b_1=b\}|\geq |\mathbb{B}_0|^{1/2}$, then define $Y_1=\{(b_1,\ldots, b_t)\in B: b_1=b\}$.  Observe that in this case, every tuple in $Y_1$ has first coordinate equal to $b$ and $|Y_1|\geq |\mathbb{B}_0|^{1/2}$.  If there is no such $b$, then note 
$$
|\mathbb{B}_0|\leq \sum_{b\in B_1}|\{(b_1,\ldots, b_s)\in B_0: b_1=b\}|\leq |B_1||\mathbb{B}_0|^{1/2}.
$$
This implies $|B_1|\geq |\mathbb{B}_0|^{1/2}$.  Let $Y_1$ consist of exactly one element of the form $(b,b_2\ldots, b_t)\in B_0$ for each $b\in B_1$.  Observe that in this case, all tuples in $Y_1$ have pairwise distinct first coordinates and $|Y_1|=|B_1|\geq |\mathbb{B}_0|^{1/2}$.  In both cases, we have defined $Y_1$ so that $|Y_1|\geq |\mathbb{B}_0|^{1/2}$ and so that $Y_1$ is indiscernible with respect to formulas of the form $\varphi(x_1,y_1)$ in the language of equality (i.e. those which only use the variable $x_1$, $y_1$).

\underline{Step i+1:} Suppose by induction we have define $Y_1\supseteq \ldots \supseteq Y_i$ such that $|Y_i|\geq |\mathbb{B}_0|^{1/2^i}$ and such that the elements in $Y_i$ are indiscernible with respect to formulas of the form $\varphi(x_1,\ldots, x_i,y_1,\ldots, y_i)$ in the language of equality.  Let 
$$
B_{i+1}=\{b\in B_0: \text{ there is }(b_1,\ldots, b_t)\in Y_i\text{ with }b=b_{i+1}\}.
$$
If there is $b\in B_{i+1}$ such that $|\{(b_1,\ldots, b_t)\in X_i: b_{i+1}=b\}|\geq |Y_i|^{1/2}$, then define $Y_{i+1}=\{(b_1,\ldots, b_t)\in X_i: b_{i+1}=b\}$.  In this case, we have $|Y_{i+1}|\geq |Y_i|^{1/2}\geq |\mathbb{B}_0|^{1/2^{i+1}}$, and every tuple in $Y_{i+1}$ has its $(i+1)$-st coordinate equal to $b$.  If there is no such $b$, then note 
$$
|Y_i|\leq \sum_{b\in B_{i+1}}|\{(b_1,\ldots, b_t)\in B_0: b_{i+1}=b\}|\leq |B_{i+1}||Y_i|^{1/2}.
$$
This implies $|B_{i+1}|\geq |Y_i|^{1/2}\geq |\mathbb{B}_0|^{1/2^{i+1}}$.  Let $Y_{i+1}$ consist of exactly one element of the form $(b_1,\ldots, b_t)\in Y_i$ with $b_{i+1}=b$ for each $b\in B_{i+1}$.  Then all tuples in $Y_{i+1}$ have distinct $(i+1)$-st coordinates and $|Y_{i+1}|=|B_{i+1}|\geq |\mathbb{B}|^{1/2^{i+1}}$.  In both cases, $|Y_{i+1}|\geq |\mathbb{B}_0|^{1/2^{i+1}}$.  Combining the definition of $Y_{i+1}$ with the inductive hypothesis implies $Y_{i+1}$ is an indiscernible set with respect to formulas of the form $\varphi(x_1,\ldots, x_{i+1}, y_1,\ldots, y_{i+1})$ in the language of equality.

At stage $t$, we obtain $Y_t\subseteq \mathbb{B}_0$ with $|Y_t|\geq |\mathbb{B}_0|^{1/2^t}$ and which is an indiscernible set with respect to formulas of the form $\varphi(x_1,\ldots, x_t, y_1,\ldots, y_t)$ in the language of equality, i.e. $Y_t$ is an indiscernible subset of $X^t$ in the language of equality.   

For part (c), we obtain the upper bound as follows.  Given $\bbar=(b_1,\ldots, b_t)$, let $\cup \bbar=\{b_1,\ldots, b_t\}$.  Then $|B|\leq \sum_{\bbar \in \mathbb{B}}|\cup \bbar|\leq \sum_{\bbar \in \mathbb{B}} t =|\mathbb{B}|t$.  For the lower bound, observe that $\mathbb{B}\subseteq B^t$ implies $|\mathbb{B}|\leq |B|^t$, so $|\mathbb{B}|^{1/t}\leq |B|$.

\end{proof}

%\bibliography{/Users/Lab/terry/Desktop/science1.bib}
\bibliography{/Users/carolineterry/Desktop/science1.bib}
%\bibliography{/Users/rickysellers/Desktop/science1.bib}
\bibliographystyle{amsplain}

\end{document}